\documentclass{amsart}
\usepackage{amscd}
\usepackage{amsmath,empheq}
\usepackage{amsfonts}
\usepackage{amssymb}
\usepackage{mathrsfs}

\usepackage[all]{xy}

\usepackage{mathtools}

\newtheorem{theorem}{Theorem}[section]

\newtheorem{prop}[theorem]{Proposition}
\newtheorem{remark}{Remark}[section]

\newtheorem{claim}{Claim} [section]

\newenvironment{proof-sketch}{\noindent{\bf Sketch of Proof}\hspace*{1em}}{\qed\bigskip}

\everymath{\displaystyle}

\newcommand{\RR}{\mathbb R}
\newcommand{\NN}{\mathbb N}

\renewcommand{\leq}{\leqslant}

\renewcommand{\geq}{\geqslant}

\begin{document}
%\hfill\today\bigskip

\title[Relaxation methods for optimal control problems]{Relaxation methods for optimal control problems}

%%%%%%%%%%%%%%%%%%%%%%%%%%%%%%%%%%%%%%%%%%%%%%%%%%%%%%%%%%%%%%%%%%%%%%%
\author[N.S. Papageorgiou]{N.S. Papageorgiou}
\address[N.S. Papageorgiou]{National Technical University, Department of Mathematics,
				Zografou Campus, Athens 15780, Greece
				 \& Institute of Mathematics, Physics and Mechanics, 1000 Ljublja\-na, Slovenia}
\email{\tt npapg@math.ntua.gr}

\author[V.D. R\u{a}dulescu]{V.D. R\u{a}dulescu}
\address[V.D. R\u{a}dulescu]{Faculty of Applied Mathematics, AGH University of Science and Technology, al. Mickiewicza 30, 30-059 Krak\'{o}w, Poland  
 \& Institute of Mathematics, Physics and Mechanics, 1000 Ljubljana, Slovenia
\& Department of Mathematics, University of Craiova, 200585 Craiova, Romania}
\email{\tt radulescu@inf.ucv.ro}

\author[D.D. Repov\v{s}]{D.D. Repov\v{s}}
\address[D.D. Repov\v{s}]{Faculty of Education and Faculty of Mathematics and Physics, University of Ljubljana \& Institute of Mathematics, Physics and Mechanics, 1000 Ljubljana, Slovenia}
\email{\tt dusan.repovs@guest.arnes.si}

\keywords{Admissible relaxation, maximal monotone map, Young measure, convex conjugate, weak norm.\\
\phantom{aa} 2010 Mathematics Subject Classification. Primary:  49J20. Secondary: 34A60.}

\begin{abstract}
We consider a nonlinear optimal control problem with dynamics described by a differential inclusion involving a maximal monotone map $A:\mathbb{R}^N\rightarrow2^{\mathbb{R}^N}$. We do not assume that $D(A)=\mathbb{R}^N$, incorporating in this way  systems with unilateral constraints in our framework. We present two relaxation methods. The first one is an outgrowth of the reduction method from the existence theory, while the second method uses Young measures. We show that the two relaxation methods are equivalent and admissible.
\end{abstract}

\maketitle

\section{Introduction}

We consider the following optimal control problem:
\begin{equation}\label{eq1}
	\left\{
		\begin{array}{ll}
			J(x,u)=\int^b_0 L(t,x(t),u(t))dt\rightarrow \inf=m\\
			\mbox{subject to:}\ -x'(t)\in A(x(t)) + f(t,x(t))u(t)\ \mbox{for almost all}\ t\in T=[0,b],\\
			x(0)=x_0, u(t)\in U(t,x(t))\ \mbox{for almost all}\ t\in T.
		\end{array}
	\right\}
\end{equation}

In this problem, the dynamics of the system are described by a differential inclusion involving a maximal monotone map $A:\RR^N\rightarrow 2^{\RR^N}$. We do not assume that $D(A)=\RR^N$ (recall that $D(A)=\{x\in\RR^N:A(x)\neq\emptyset\}$ is the domain of $A(\cdot)$). This way we incorporate in our framework systems with unilateral constraints (differential variational inequalities). In addition, the control constraint set $U(t,x)$ is state-dependent, that is, the system has {\it a priori} feedback, a setting which is of interest in engineering and economic problems. The existence theory of such problems is based on the so-called ``reduction technique", which was developed in the pioneering works of Cesari \cite{10, 11}  and Berkovitz \cite{3} (see also the books of Berkovitz \cite{4} and Cesari \cite{12}). According to this method, the original optimal control problem is reduced to a calculus of variations problem with multivalued dynamics. This problem is obtained by elimination of the control variable $u\in\RR^m$. For this approach to work, we need to have enough convex structure in the problem, usually expressed in terms of the ``property Q" of Cesari. In the absence of such a convex structure, a minimizing sequence of state-control pairs need not converge to an admissible pair. To rectify this, we need to augment the system and pass to a ``convexified" version known as the ``relaxed problem", which captures the asymptotic behaviour of the minimizing sequences. The process of relaxation is a delicate one since we have to strike a sensitive balance between competing requirements. We want that the relaxed problem exhibits the following three fundamental properties:
\begin{itemize}
	\item [(a)] Every original state is also a relaxed state (that is, the original problem is embedded in the relaxed one).
	\item [(b)] Every relaxed state can be approximated by original ones (that is, we want to make sure that we did not augment the system too much).
	\item [(c)] The values of the relaxed and original problems are equal and the relaxed problem has a solution (that is, there exists an optimal state-control pair).
\end{itemize}

Note that the first two requirements concern the dynamics of the system, while the third one concerns the cost functional. Any relaxation method which meets these three requirements, is said to be ``admissible".

In this paper, under general conditions on the data of problem \eqref{eq1}, we present two such admissible relaxation methods. The first one is an outgrowth of the reduction method from the existence theory, while the second method uses Young measures.

Relaxation methods for different classes of optimal control problems can be found in the works of Avgerinos \& Papageorgiou \cite{1}, Buttazzo \cite{8}, Buttazzo \& Dal Maso \cite{9}, Emamizadeh, Farjudian \& Mikayelyan \cite{14}, Hu \& Papageorgiou \cite{21}, Liu, Liu \& Fu \cite{22},  Papageorgiou, R\u adulescu \& Repov\v{s} \cite{PRRANA, PRREECT0, PRRAMO, PRRJDE, PRREECT}, Roubicek \cite{24}, Tolstonogov \cite{25}, and Warga \cite{27}. However, none of the aforementioned works covers the case of systems with unilateral constraints.

\section{Mathematical background and hypotheses}

Let $A:\RR^N\rightarrow2^{\RR^N}$. The domain of $A(\cdot)$ is the set
\begin{equation*}
	D(A) = \{x\in\RR^N: A(x)\neq\emptyset\}
\end{equation*}
and the graph of $A(\cdot)$ is the set
\begin{equation*}
	{\rm Gr}\,A = \{(x,x^*)\in\RR^N\times\RR^N:x^*\in A(x)\}.
\end{equation*}

We say that $A(\cdot)$ is ``monotone", if
\begin{equation*}
	(x^*-y^*,x-y)_{\RR^N}\geq0\ \mbox{for all}\ (x, x^*),\ (y,y^*)\in {\rm Gr}\,A.
\end{equation*}

If $A(\cdot)$ satisfies
\begin{equation*}
	(x^*-y^*,x-y)_{\RR^N}>0\ \mbox{for all}\ (x,x^*),(y,y^*)\in {\rm Gr}\,A\ \mbox{with}\ x\neq y,
\end{equation*}
then we say that $A(\cdot)$ is ``strictly monotone". We say that $A(\cdot)$ is ``maximal monotone", if ${\rm Gr}\,A$ is maximal with respect to inclusion among the graphs of all monotone maps. This is equivalent to saying that
\begin{equation*}
	``\mbox{if}\ (x^*-y^*,x-y)_{\RR^N}\geq0\ \mbox{for all}\ (x,x^*)\in {\rm Gr}\,A,\ \mbox{then}\ (y, y^*)\in {\rm Gr}\,A".
\end{equation*}

Suppose that $Y,Z$ are Banach spaces and $V:Y\rightarrow Z$. We say that $V(\cdot)$ is ``compact", if it is continuous and maps bounded sets in $Y$ onto relatively compact subsets of $Z$. Also, we say that $V(\cdot)$ is ``completely continuous", if $y_n\xrightarrow{w}y$ in $Y$, implies that $V(x_n)\rightarrow V(x)$ in $Z$. In general, these two notions are distinct. However, if $Y$ is reflexive, then complete continuity implies compactness. Moreover, if in addition, $V(\cdot)$ is linear, then the two notions coincide.

From fixed point theory, we will need the so-called ``Leray-Schauder alternative theorem", which we recall here.

\begin{theorem}\rm{(See e.g. \cite{17, PRRbook})}\label{th1}
	If $Y$ is a Banach space, $V:Y\rightarrow Y$ is a compact map, and
	\begin{equation*}
		S=\{y\in Y: y=\lambda V(y)\ \mbox{for some}\ 0<\lambda<1\},
	\end{equation*}
	then one of the following two statements is true:
	\begin{itemize}
		\item [(a)] $S$ is bounded;
		\item [(b)] $V$ has a fixed point.
	\end{itemize}
\end{theorem}

Now let $(\Omega, \Sigma, \mu)$ be a finite measure space and $X$  a separable Banach space. We introduce the following families of subsets of $X$:
\begin{equation*}
	\begin{array}{ll}
		& P_{f(c)}(X) = \{A\subseteq X: A\ \mbox{is nonempty, closed (and convex)}\}, \\
		& P_{(w)k(c)}(X) = \{A\subseteq X: A\ \mbox{is nonempty, $(w\mbox{-})$compact (and convex)}\}.
	\end{array}
\end{equation*}

A multifunction (set-valued function) $F:\Omega\rightarrow P_f(X)$ is said to be ``measurable", if for all $x\in  X$ the function
\begin{equation*}
	\omega\mapsto d(x, F(\omega))=\inf\left\{||x-v||_X:v\in F(\omega)\right\}
\end{equation*}
is measurable. If $F(\cdot)$ is measurable, then
\begin{equation*}
	{\rm Gr}\,F=\{(\omega,v)\in\Omega\times X:x\in F(\omega)\}\in\Sigma\otimes B(X)
\end{equation*}
and the converse is true if $\Sigma$ is $\mu$-complete. In general, a multifunction $F:\Omega\rightarrow2^X\backslash\{\emptyset\}$ is said to be ``graph measurable", if ${\rm Gr}\,F\in\Sigma\otimes B(X)$.

Given $1\leq p\leq\infty$ and a multifunction $F:\Omega\rightarrow 2^X\backslash\{\emptyset\}$, we introduce the set
\begin{equation*}
	S^p_F=\{f\in L^p(\Omega, X):f(\omega)\in F(\omega)\ \mu\mbox{-a.e.}\ \mbox{in}\ \Omega\}.
\end{equation*}
This set can be empty. For a graph measurable multifunction $F(\cdot)$, $S^p_F\neq\emptyset$ if and only if $w\mapsto \inf\left\{||v||_X:v\in F(\omega)\right\}$ belongs in $L^p(\Omega)$. The set $S^p_F$ is ``decomposable", in the sense that for every $(A,f,g)\in\Sigma\times S^p_F\times S^p_F$ we have
\begin{equation*}
	\chi_Af + \chi_{A^c}g\in S^p_F.
\end{equation*}
Here, for every $C\in\Sigma,\ \chi_C$ is the characteristic function of $C$ and $C^c$ is the complement of $C$ (that is, $C^c=X\backslash C$). Since $\chi_{A^c}=1-\chi_A$, the notion of decomposability formally looks like that of convexity. Only now the coefficients are functions. Nevertheless, decomposability is a good substitute of convexity and several results valid for convex sets have their counterparts for decomposable sets (see Fryszkowski \cite{16}).

Let $Y$ be Hausdorff topological space and $G:Y\rightarrow 2^X\backslash\{\emptyset\}$ a multifunction. We say that $G(\cdot)$ is ``lower semicontinuous" (lsc for short), resp. ``upper semicontinuous" (usc for short), if for every $U\subseteq X$ open, the set
\begin{equation*}
	\begin{array}{ll}
		F^-(U) = \{y\in Y:F(y)\cap U\neq\emptyset\}\\
		\mbox{resp.}\ F^+(U)=\{y\in Y:F(y)\subseteq U\}
	\end{array}
\end{equation*}
is open.

Recall that on $P_f(X)$ we can define a generalized metric, known as the ``Hausdorff metric", by
\begin{equation*}
	h(A,C)=\max\left\{\sup\{d(a,C):a\in A\},\ \sup\{d(c,A):c\in C\}\right\}
\end{equation*}
for all $A,C\in P_f(X)$. A multifunction $F:Y\rightarrow P_f(X)$ is said to be ``$h$-continuous", if it is continuous from $Y$ into the metric space $(P_f(X),h)$.

The next theorem, due to Bressan \& Colombo \cite{6} and Fryszkowski \cite{15}, is an illustration of how decomposability can serve as a substitute of convexity. It extends the celebrated Michael selection theorem.

\begin{theorem}\rm{ (See e.g. \cite{20})}\label{th2}
	If $Y$ is a separable metric space and $F:Y\rightarrow P_f(L^1(\Omega,X))$ is a lsc multifunction with decomposable values, then there exists a continuous map $g:Y\rightarrow L^1(\Omega,X)$ such that
	\begin{equation*}
		g(y)\in F(y)\ \mbox{for all}\ y\in Y.
	\end{equation*}
\end{theorem}

Now, let $\Omega=T=[0,b]$. On $L^1(T,X)$ we introduce the ``weak norm" defined by
\begin{equation*}
	||u||_w=\sup\left\{||\int^t_s u(\tau)d\tau||_X:0\leq s\leq t\leq b\right\}
\end{equation*}
or, equivalently, by
\begin{equation*}
	||u||_w=\sup\left\{||\int^t_0 u(\tau)d\tau||_X:0\leq t\leq b\right\}.
\end{equation*}

From Hu \& Papageorgiou \cite[p. 24]{21}, we have the following result.

\begin{prop}\rm{(See \cite{21})}\label{prop3}
		If $X$ is reflexive, $\{u_n,u\}_{n\geq1}\subseteq L^p(T,X),\ 1\leq p<\infty$, $ u_n\xrightarrow{||\cdot||_w}u$ and
	\begin{equation*}
		\sup_{n\geq1}||u_n||_p<\infty,
	\end{equation*}
	then $u_n\xrightarrow{w}u\ \mbox{in}\ L^p(T,X)$.
\end{prop}

Let $(\Omega,\Sigma,\mu)$ be a complete finite measure space and $Y$ a Polish space. Recall that this means that $Y$ is a separable Hausdorff topological space and there is a metric $d$ on $Y$ compatible with the topology of $Y$ such that $(Y,d)$ is complete. By $P(Y)$ we denote the set of all probability measures on $Y$ endowed with the narrow topology $\tau_n$ (see Papageorgiou \& Winkert \cite[p. 375]{23}). Let $B(Y)$ be a Borel $\sigma$-field on $Y$, $ca(\Sigma\otimes B(Y)$  the space of all $\RR$-valued signed measures on $\Sigma\otimes B(Y)$ and $p_\Omega:\Omega\times Y\rightarrow\Omega$ the projection map. Given $\lambda\in ca(\Sigma\otimes B(Y))$ and if $\mu=\lambda\circ p^{-1}_\Omega$, 
 the disintegration theorem says that there exists a $\Sigma$-measurable map $\hat\lambda:\Omega\rightarrow P(Y)$ such that
\begin{equation*}
	\lambda(A\times C) = \int_A\hat\lambda(\omega)(C)d\mu\ \mbox{for all}\ A\in\Sigma,\  C\in B(Y).
\end{equation*}

By $ca_+(\Sigma\otimes B(Y))$ we denote the $\RR_+$-valued elements of $ca(\Sigma\times B(Y))$.

A ``Young measure" on $\Omega\times Y$ is a $\lambda\in ca_+(\Sigma\otimes B(Y))$ such that
\begin{equation*}
	\mu = \lambda\circ p^{-1}_\Omega
\end{equation*}
(that is, $\mu(A)=\lambda(A\times Y)$ for all $A\in\Sigma$). The space of Young measures on $\Omega\times Y$ is denoted by ${\mathcal Y}(\Omega\times Y)$. On account of the disintegration theorem mentioned above, we can identify $\lambda\in {\mathcal Y}(\Omega\times Y)$ with its disintegration $\hat\lambda(\cdot)$. So, we say that a Young measure is a measurable map $\hat\lambda:\Omega\rightarrow P(Y)$. Such maps are also known as ``transition measures". The space of transition measures is denoted by $R(\Omega, Y)$. We know that the following statements are equivalent:
\begin{itemize}
	\item [(a)] $\hat\lambda\in\RR(\Omega,Y)$.
	\item [(b)] For every $C\in B(Y),$ the map $\omega\mapsto\xi_C(\omega)=\hat\lambda(\omega)(C)$ is $\Sigma$-measurable (see Papageorgiou \& Winkert \cite[p. 387]{23}). Given a $\Sigma$-measurable function $u:\Omega\rightarrow Y$, the ``Young measure associated with $u$", is the transition probability defined by
		\begin{equation*}
			\hat\lambda^u(\omega) = \delta_{u(\omega)}
		\end{equation*}
		with $\delta_{u(\omega)}(\cdot)$ being the Dirac measure defined by
		\begin{equation*}
			\delta_{u(\omega)}(C)=\left\{
				\begin{array}{ll}
					1 & \mbox{if}\ u(\omega)\in C\\
					0 & \mbox{otherwise}
				\end{array}
			\right.
			\mbox{for all}\ C\in B(Y).
		\end{equation*}
\end{itemize}

Let $\varphi:\Omega\times Y\rightarrow\RR$ be a ``Carath\'eodory function", that is, for all $y\in Y$ the mapping $\omega\mapsto \varphi(\omega,y)$ is $\Sigma$-measurable and for $\mu$-a.e. $\omega\in\Omega$ the mapping $y\mapsto\varphi(\omega,y)$ is continuous. We know that such function is $\Sigma\otimes B(Y)$-measurable (see Hu \& Papageorgiou \cite[p. 142]{20}). We say that a Carath\'eodory function is ``$L^1$-bounded", if there exists $k\in L^1(\Omega)$ such that $|\varphi(\omega,y)|\leq k(\omega)$ $\mu$-a.e. on $\Omega$, for all $y\in Y$. By ${\rm Car}_b(\Omega\times Y)$ we denote the space of all $L^1$-bounded Carath\'eodory functions.

The ``Young narrow topology" on ${\mathcal Y}(\Omega\times Y)$, is the weakest topology on ${\mathcal Y}(\Omega\times Y)$ for which the maps
\begin{equation*}
	\lambda\mapsto I_\varphi(\lambda)=\int_{\Omega\times Y}\varphi(\omega,y) d\lambda = \int_\Omega\left(\int_Y\varphi(\omega,y)\hat\lambda(\omega)(dy)\right)d\mu
\end{equation*}
with $\hat\lambda(\cdot)$ being the disentegration of $\lambda$ and $\varphi\in {\rm Car}_b(\Omega\times Y)$, are all continuous. This topology on ${\mathcal Y}(\Omega\times Y)$ is denoted by $\tau^y_n$.

Now we introduce the hypotheses on the data of problem (\ref{eq1}).

\smallskip
$H(A)$: $A:D(A)\subseteq\RR^N\rightarrow2^{\RR^N}$ is a maximal monotone map such that $0\in A(0)$.

\begin{remark}\label{rem1}
	We do not require that $D(A)=\RR^N$. In this way our framework incorporates systems with unilateral constraints (differential variational inequalities).
\end{remark}

$H(f)$: $f:T\times\RR^N\rightarrow\mathcal{L}(\RR^m,\RR^N)$ is a map such that
\begin{itemize}
	\item [(i)] for all $x\in\RR^N$ and all $u\in\RR^m,\ t\mapsto f(t,x)u$ is measurable;
	\item [(ii)] for every $r>0$, there exists $l_r\in L^1(T)$ such that
		\begin{equation*}
			||f(t,x) - f(t,y)||_\mathcal{L} \leq l_r(t)|x-y|\ \mbox{for almost all}\ t\in T  \ \mbox{and all}\ |x|,|y|\leq r;
		\end{equation*}
	\item [(iii)] $||f(t,x)||_\mathcal{L} \leq a(t)(1+|x|)\ \mbox{for almost all}\ t\in T  \ \mbox{and all}\ x\in\RR^N,\ \mbox{with}\ a\in L^1(T).$
\end{itemize}

\smallskip
$H(U)$: $U:T\times\RR^N\rightarrow P_k(\RR^m)$ is a multifunction such that
\begin{itemize}
	\item [(i)] for all $x\in\RR^N,\ t\mapsto U(t,x)$ is measurable;
	\item [(ii)] there exists $k\in L^\infty(T)$ such that
		\begin{equation*}
			h(U(t,x), U(t,y)) \leq k(t)|x-y|\ \mbox{for almost all}\ t\in T  \ \mbox{and all}\ x,\,y\in\RR^N;
		\end{equation*}
	\item [(iii)] $|U(t,x)|=\sup\{|u|:u\in U(t,x)\}\leq a_0(t)$ for almost all $t\in T$   and all $x\in\RR^N$, with $a_0\in L^\infty(T)$.
\end{itemize}

\smallskip
$H(L)$: $L:T\times\RR^N\times\RR^m\rightarrow\RR$ is a function such that
\begin{itemize}
	\item [(i)] for all $(x,u)\in\RR^N\times\RR^m,\ t\mapsto L(t,x,u)$ is measurable;
	\item [(ii)] for every $r>0$, there exists $\vartheta_r\in L^1(T)$ such that
		\begin{equation*}
			|L(t,x,u) - L(t,y,v)|\leq \vartheta_r(t)(|x-y| + |u-v|)
		\end{equation*}
		for almost all $t\in T$   and all $|x|,\,|y|,\,|u|,\,|v|\leq r$;
	\item [(iii)] for every $r>0$, there exists $a_r\in L^1(T)$ such that
		\begin{equation*}
			|L(t,x,u)|\leq a_r(t)\ \mbox{for almost all}\ t\in T   \ \mbox{and all}\ |x|,\,|u|\leq r.
		\end{equation*}
\end{itemize}

\smallskip
$H_0$: $x_0\in\overline{D(A)}$.

\smallskip
We introduce the ``convexified" dynamics of problem (\ref{eq1}), namely the following control system
\begin{equation}\label{eq2}
	\left\{
		\begin{array}{ll}
			-x'(t)\in A(x(t)) + f(t,x(t))u(t)\ \mbox{for almost all}\ t\in T=[0,b],\\
			x(0)=x_0, u(t)\in {\rm conv}\,  U(t,x(t))\ \mbox{for almost all}\ t\in T.
		\end{array}
	\right\}
\end{equation}

Then we define the following two sets:
\begin{equation*}
	\begin{array}{ll}
		P=\{(x,u)\in C(T,\RR^N)\times L^1(T,\RR^m):\ \mbox{$(x,u)$ is an admissible state-control pair for (\ref{eq1})}\},\\
		P_c=\{(x,u)\in C(T,\RR^N)\times L^1(T,\RR^m):\ \mbox{$(x,u)$ is an admissible state-control pair for (\ref{eq2})}\}.
	\end{array}
\end{equation*}

Also we set
\begin{equation*}
	\mathscr{S}={\rm proj}_{C(T,\RR^N)}P\ \mbox{and}\ \mathscr{S}_c={\rm proj}_{C(T,\RR^N)}P_c.
\end{equation*}

These are the sets of admissible trajectories for the original system (the set $\mathscr{S}$) and for the convexified system (the set $\mathscr{S}_c$).

\section{First relaxation method}

The first relaxation method is motivated by the ``reduction method" of the existence theory and it uses the convexified control system (\ref{eq2}).

\begin{prop}\label{prop4}
	If hypotheses $H(A), H(f), H(U), H_0$ hold, then $P\neq\emptyset$ and there exists $\hat{c}>0$ such that $||x||_{C(T,\RR^N)}\leq\hat{c}$ for all $x\in \mathscr{S}_c$.
\end{prop}

\begin{proof}
	Consider the following orientor field
	\begin{equation*}
		F(t,x) = f(t,x)U(t,x) = \bigcup_{u\in U(t,x)} f(t,x)u\in P_k(\RR^N).
	\end{equation*}
	
	\begin{claim}\label{claim1}
		For every $x\in\RR^N$, the map $ t\mapsto F(t,x)$ is measurable.
	\end{claim}
	
	We fix $x\in\RR^N$ and consider the multifunction $t\rightarrow U(t,x)$. The measurability of this multifunction (see hypothesis $H(U)(i)$) implies that we can find a sequence $\{u^x_n\}_{n\geq1}\subseteq S_{U(\cdot,x)}=\{u:T\rightarrow\RR^m$ measurable and $u(t)\in U(t,x)$ for almost all $t\in T\}$ such that
	\begin{equation*}
		U(t,x)=\{\overline{u^x_n(t)}\}_{n\geq1}\ \mbox{for almost all}\ t\in T
	\end{equation*}
	(see Hu \& Papageorgiou \cite[Theorem 24, p.156]{20}). Then we have
	\begin{equation*}
		\begin{array}{ll}
			& F(t,x) = f(t,x)\{\overline{u^x_n(t)}\}_{n\geq1} = \{\overline{f(t,x)u^x_n(t)}\}_{n\geq1},\\
			\Rightarrow & t\mapsto F(t,x)\ \mbox{is measurable}.
		\end{array}
	\end{equation*}
	
	This proves Claim \ref{claim1}.

	\begin{claim}\label{claim2}
		For almost all $t\in T$, the map $ x\mapsto F(t,x)$ is locally $h$-Lipschitz.
	\end{claim}

	Let $r>0$ and consider $x,y\in\RR^N$ such that $|x|,|y|\leq r$. If $\xi\in F(t,x)$, then
	\begin{equation*}
		\xi = f(t,x)u\ \mbox{where}\ u\in U(t,x).
	\end{equation*}

	Given $\epsilon>0$, we choose $v\in U(t,y)$ such that
	\begin{equation}\label{eq3}
		\begin{array}{ll}
			|u-v| & \leq d(u,U(t,y)) + \epsilon\\
				& \leq h(U(t,x), U(t,y)) +\epsilon\\
				& \leq k(t)|x-y| +\epsilon\ \mbox{for almost all}\ t\in T\\
				& \mbox{(see hypothesis $H(U)(ii)$)}.
		\end{array}
	\end{equation}

	Then we have
	\begin{equation*}
		\begin{array}{lll}
			d(\xi, F(t,y)) & =& d(f(t,x)u, F(t,y)) \\
			& \leq& |f(t,x)u -  f(t,y)v| \\
			& \leq& ||f(t,x)||_\mathcal{L}|u-v| + |f(t,x)v-f(t,y)v| \\
			& \leq& a_r(t)|u-v| + l_r(t)|x-y|\ (\mbox{with}\ a_r(\cdot)=a(\cdot)(1+r)) \\
			& \leq& [a_r(t)k(t) + l_r(t)]|x-y| + a_r(t)\epsilon\ \mbox{for almost all}\ t\in T\ \mbox{(see (\ref{eq3}))} \\
			 &\Rightarrow & h(F(t,x), F(t,y)) \leq k_r(t)|x-y| + a_r(t)\epsilon\\
			&& \mbox{for almost all}\ t\in T,\ \mbox{where}\ k_r\in L^1(T).
		\end{array}
	\end{equation*}

	Let $\epsilon\rightarrow0^+$, and
	conclude that for almost all $t\in T$, the multifunction $x\rightarrow F(t,x)$ is locally $h$-Lipschitz. This proves Claim \ref{claim2}.

\smallskip
	From Claims \ref{claim1} and \ref{claim2} it follows that $(t,x)\rightarrow F(t,x)$ is measurable (see Hu \& Papageorgiou \cite[Proposition 7.9, p.229]{20}). Thus, we can infer that for every measurable function $x:T\rightarrow\RR^N$, the multifunction $t\rightarrow F(t, x(t))$ is measurable (superpositional measurability).

	We consider the following multivalued Cauchy problem:
	\begin{equation}\label{eq4}
		-x'(t)\in A(x(t)) + F(t,x(t))\ \mbox{for almost all}\ t\in T,\ x(0)=x_0.
	\end{equation}

	We will prove
	 the existence of solutions for this problem. To this end, let $h\in L^1(T,\RR^N)$ and consider the following Cauchy problem:
	\begin{equation}\label{eq5}
		-x'(t)\in A(x(t)) + h(t)\ \mbox{for almost all}\ t\in T,\ x(0)=x_0.
	\end{equation}

	By virtue of the B\'enilan-Brezis theorem (see Brezis \cite[Proposition 3.8, p.82]{7}), we know that problem (\ref{eq5}) has a unique solution $x\in W^{1,1}((0,b),\RR^N)=AC^1(T,\RR^N)$. So, we can define the solution map $K:L^1(T,\RR^N)\rightarrow C(T,\RR^N)$, which to each $h\in L^1(T,\RR^N)$ assigns the unique solution $K(h)\in AC^1(T,\RR^N)\subseteq C(T,\RR^N)$ of (\ref{eq5}).

	\begin{claim}\label{claim3}
		The solution map $K:L^1(T,\RR^N)\rightarrow C(T,\RR^N)$ is completely continuous.
	\end{claim}

	Suppose that $h_n\xrightarrow{w}h$ in $L^1(T,\RR^N)$ and let $x_n=K(h_n)$ for all $n\in\NN$ and $x=K(h)$. From Baras \cite{2}, we know that
	\begin{equation*}
		\{x_n\}_{n\geq1}\subseteq C(T,\RR^N)\ \mbox{is relatively compact}.
	\end{equation*}

	So, up to a subsequence, we may assume that
	\begin{equation}\label{eq6}
		x_n\rightarrow x^*\ \mbox{in}\ C(T,\RR^N).
	\end{equation}

	Exploiting the monotonicity of $A(\cdot)$, we obtain
	\begin{equation*}
		\begin{array}{ll}
			& (x'_n(t) - x'(t), x_n(t)-x(t))_{\RR^N} + (h_n(t) - h(t), x_n(t)-x(t))_{\RR^N}\leq0\ \mbox{for almost all}\ t\in T,\\
			\Rightarrow & \frac{1}{2}\frac{d}{dt}|x_n(t)-x(t)|^2\leq(h_n(t)-h(t),x(t)-x_n(t))_{\RR^N}\ \mbox{for almost all}\ t\in T,\\
			\Rightarrow & |x_n(t)-x(t)|^2\leq 2\int^t_0(h_n(s)-h(s),x(s)-x_n(s))_{\RR^N}ds\ \mbox{for all}\ t\in T,\ n\in\NN, \\
			\Rightarrow & |x^*(t)-x(t)|^2\leq0\ \mbox{for all}\ t\in T\ \mbox{(see (\ref{eq6}) and recall that $h_n\xrightarrow{w}h$ in $L^1(T,\RR^N)$)}, \\
			\Rightarrow & x=x^*.
		\end{array}
	\end{equation*}

	We deduce that for the original sequence we have
	\begin{equation*}
		\begin{array}{ll}
			& x_n\rightarrow x\ \mbox{in}\ C(T,\RR^N), \\
			\Rightarrow & K:L^1(T,\RR^N) \rightarrow C(T,\RR^N)\ \mbox{is completely continuous}.
		\end{array}
	\end{equation*}

	This proves Claim \ref{claim3}.

\smallskip
	Let $N_F:C(T,\RR^N)\rightarrow2^{L^1(T,\RR^N)}$ be the multivalued Nemitsky operator corresponding to the multifunction $F(t,x)$, that is,
	\begin{equation*}
		N_F(x) = S^1_{F(\cdot,x(\cdot)))}\ \mbox{for all}\ x\in C(T,\RR^N).
	\end{equation*}

	The measurability of $t\rightarrow F(t,x(t))$ and hypotheses $H(f)(iii)$ and $H(U)(iii)$ imply that
	\begin{equation*}
		N_F(x)\in P_{wk}(L^1(T,\RR^N))\ \mbox{for all}\ x\in C(T,\RR^N).
	\end{equation*}

	On account of Theorems 7.26 and 7.27 of Hu \& Papageorgiou \cite[p. 237]{20}, the multifunction $x\rightarrow N_F(x)$ is $h$-continuous. In particular, it is then also lsc (see Proposition 2.66 of Hu \& Papageorgiou \cite[p.61]{20}). So, we can apply Theorem \ref{th2} and produce a continuous map $g:C(T,\RR^N)\rightarrow L^1(T,\RR^N)$ such that
	\begin{equation}\label{eq7}
		g(x)\in N_F(x)\ \mbox{for all}\ x\in C(T,\RR^N).
	\end{equation}

	Evidently, the map
	\begin{equation}\label{eq8}
		x\rightarrow (K\circ g)(x)\ \mbox{is continuous}.
	\end{equation}

	Hypotheses $H(f)(iii)$ and $H(U)(iii)$ and the complete continuity of $K(\cdot)$, imply that $(K\circ g)(\cdot)$ maps bounded sets in $C(T,\RR^N)$ to relatively compact sets in $C(T,\RR^N)$. Therefore
	\begin{equation}\label{eq9}
		x\rightarrow(K\circ g)(x)\ \mbox{is compact}.
	\end{equation}

	Consider the set
	\begin{equation*}
		C=\{x\in C(T,\RR^N): x=\lambda(K\circ g)(x),\ 0<\lambda<1\}.
	\end{equation*}

	\begin{claim}\label{claim4}
		The set $C\subseteq C(T,\RR^N)$ is bounded.
	\end{claim}

	Let $x\in C$. Then
	\begin{equation*}
		\begin{array}{ll}
			& \frac{1}{\lambda}x = K(g(x)),\\
			\Rightarrow & -x'(t)\in\lambda A(\lambda^{-1}x(t)) + \lambda g(x)(t)\ \mbox{for almost all}\ t\in T,\ x(0)=\lambda x_0.
		\end{array}
	\end{equation*}

	We take inner product with $x(t)$ and use hypothesis $H(A)$. Then
	\begin{equation*}
		\begin{array}{ll}
			& \frac{1}{2}\frac{d}{dt}|x(t)|^2 + \lambda(g(x)(t),x(t))_{\RR^N} \leq 0\ \mbox{for almost all}\ t\in T, \\
			\Rightarrow & \frac{1}{2}|x(t)|^2 \leq \frac{1}{2}|x_0|^2 + \int^t_0|g(x)(s)||x(s)|ds\ \mbox{(since $0<\lambda<1$)}.
		\end{array}
	\end{equation*}

	Invoking Lemma A.5 of Brezis \cite[p. 157]{7}, we obtain
	\begin{equation*}
		\begin{array}{ll}
			&\displaystyle |x(t)|  \leq |x_0| + \int^t_0|g(x)(s)|ds \leq\\
			& \displaystyle |x_0| + \int^t_0\hat{a}(s)(1+|x(s)|)ds\ \mbox{for all}\ t\in T,\ \mbox{with}\ \hat{a}\in L^1(T) \\
			& \mbox{(see (\ref{eq7}) and hypotheses $H(F)(iii), H(U)(iii)$)}, \\
			\Rightarrow & |x(t)| \displaystyle \leq c_1\ \mbox{for all}\ t\in T,  x\in C\ \mbox{and some}\ C_1>0\ \mbox{(by Gronwall's inequality)}.
		\end{array}
	\end{equation*}

	Therefore $C\subseteq C(T,\RR^N)$ is bounded. This proves Claim \ref{claim4}.

\smallskip
	Then (\ref{eq9}) and Claim \ref{claim4} permit the use of Theorem \ref{th1} (the Leray-Schauder alternative theorem). So, we can find $\hat{x}\in C(T,\RR^N)$ such that
	\begin{equation*}
		\begin{array}{ll}
			&\hat{x}=(K\circ g)(\hat{x}), \\
			\Rightarrow & -\hat{x}'(t)\in A(\hat{x}(t)) + g(\hat{x})(t)\in A(\hat{x}(t)) + F(t,\hat{x}(t))\ \mbox{for almost all}\ t\in T, \\
			& \hat{x}(0)=x_0.
		\end{array}
	\end{equation*}

	Consider the multifunction
	\begin{equation*}
		\begin{array}{ll}
			& E(t) = \{u\in U(t,\hat{x}(t)):g(\hat{x})(t)=f(t,\hat{x}(t))u\},\\
			\Rightarrow & {\rm Gr}\,E\in \mathcal{L}_T\otimes B(\RR^m),
		\end{array}
	\end{equation*}
	with $\mathcal{L}_T$ being the Lebesgue $\sigma$-field of $T$ and $B(\RR^m)$ the Borel $\sigma$-field of $\RR^m$. Applying the Yankov-von Neumann-Aumann selection theorem (see Hu \& Papageorgiou \cite[Theorem 2.14, p.158]{20}), we obtain a measurable map $\hat{u}:T\rightarrow\RR^m$ such that
	\begin{equation*}
		\begin{array}{ll}
			& \hat{u}(t)\in E(t)\ \mbox{for almost all}\ t\in T, \\
			\Rightarrow & g(\hat{x})(t) = f(t,\hat{x}(t))\hat{u}(t)\ \mbox{for almost all}\ t\in T, \\
			\Rightarrow & (\hat{x},\hat{u})\in P.
		\end{array}
	\end{equation*}

	Moreover, as in the proof of Claim \ref{claim4}, we can show that there exists $\hat{c}>0$ such that
	\begin{equation*}
		||x||_{C(T,\RR^N)}\leq\hat{c}\ \mbox{for all}\ x\in \mathscr{S}_C.
	\end{equation*}
The proof  of Proposition 3.1 is now complete.
\end{proof}

In what follows, we denote by $L^1(T,\RR^m)_w$  the Lebesgue space $L^1(T,\RR^N)$ equipped with the weak topology and by $L^1_w(T,\RR^N)$ the same space furnished with the weak norm.

\begin{prop}\label{prop5}
	If hypotheses $H(A),H(f),H(U),H_0$ hold then $P_c\subseteq C(T,\RR^N)\times L^1(T,\RR^m)_w$ is sequentially compact.
\end{prop}

\begin{proof}
	Let $\{(x_n,u_n)\}_{n\geq1}\subseteq P_c$. Hypothesis $H(U)(iii)$ implies that by passing to a subsequence if necessary, we may assume that
	\begin{equation}\label{eq10}
		u_n\xrightarrow{w} u\ \mbox{in}\ L^1(T,\RR^m).
	\end{equation}
	
	Recall that $||x_n||_{C(T,\RR^N)}\leq \hat{c}$ for all $n\in\NN$. Hence by hypothesis $H(f)(iii)$ and the Dunford-Pettis theorem, we see that
	\begin{equation*}
		\{g_n(\cdot)=f(\cdot,x_n(\cdot))u_n(\cdot)\}_{n\geq1}\subseteq L^1(T,\RR^n)
	\end{equation*}
	is relatively $w$-compact. So, we may assume that for some $g\in L^1(T,\RR^N)$ we have
	\begin{equation}\label{eq11}
		g_n\xrightarrow{w} g\ \mbox{in}\ L^1(T,\RR^N).
	\end{equation}
	
	By  Claim \ref{claim4} in the proof of Proposition \ref{prop4}, we have
	\begin{equation}\label{eq12}
		x_n=K(g_n)\rightarrow K(g)=x\ \mbox{in}\ C(T,\RR^N).
	\end{equation}
	
	For every $h\in L^\infty(T,\RR^N)=L^1(T,\RR^N)^*$, we have
	\begin{equation}\label{eq13}
		\int^b_0(g_n,h)_{\RR^N}dt = \int^b_0(u_n,f(t,x_n)^*h)_{\RR^m}dt.
	\end{equation}
	
	Note that
	\begin{eqnarray}\label{eq14}
		&& ||f(t,x_n(t))^* - f(t,x(t))^*||_\mathcal{L} \nonumber \\
		&= & ||f(t,x_n(t)) - f(t,x(t))||_\mathcal{L} \leq l_{\hat{c}}(t)||x_n-x||_{C(T,\RR^N)}\ \mbox{for almost all}\ t\in T, \nonumber \\
		\Rightarrow && ||f(\cdot,x_n(\cdot))^*h(\cdot) - f(\cdot,x(\cdot))^*h(\cdot)||_\infty\rightarrow0\ \mbox{as}\ n\rightarrow\infty\ \mbox{(see (\ref{eq12}))}.
	\end{eqnarray}
	
	If in (\ref{eq13}) we pass to the limit as $n\rightarrow\infty$ and use (\ref{eq10}), (\ref{eq11}), (\ref{eq14}), then
	\begin{eqnarray}\label{eq15}
		&& \int^b_0(g,h)_{\RR^N}dt = \int^b_0(f(t,x)u,h)_{}\RR^Ndt\ \mbox{for all}\ h\in L^\infty(T,\RR^N), \nonumber \\
		&\Rightarrow & g(\cdot)=f(\cdot,x(\cdot))u(\cdot).
	\end{eqnarray}
	
	From Proposition 3.6 of Brezis \cite[p. 70]{7}, we have
	\begin{equation*}
		\begin{array}{ll}
			& \frac{1}{2}|x_n(t)-y|^2\leq \frac{1}{2}|x_n(s)-y|^2 + \int^t_s(g_n(\tau)-y^*,x_n(\tau)-y)_{\RR^N}dt \\
			& \mbox{for all}\ 0\leq s\leq t\leq b,\ (y,y^*)\in {\rm Gr}\,A,\ n\in\NN, \\
			\Rightarrow & \frac{1}{2}|x(t)-y|^2\leq\frac{1}{2}|x(s)-y|^2 + \int^t_s(g(\tau)-y^*, x(\tau)-y)_{\RR^N}dt\\
			& \mbox{for all}\ 0\leq s\leq t\leq b,\ (y,y^*)\in {\rm Gr}\,A.
		\end{array}
	\end{equation*}
	
	Invoking once again Proposition 3.6 of Brezis \cite[p. 70]{7}, we have
	\begin{eqnarray*}
		-x'(t)\in A(x(t)) + f(t,x(t))u(t)\ \mbox{for almost all}\ t\in T, \\
		x(0)=x_0,\ u(t)\in {\rm conv}\, U(t,x(t))\ \mbox{for almost all}\ t\in T,
	\end{eqnarray*}
	(see (\ref{eq15}) and hypothesis $H(U)(ii)$). Therefore
	\begin{equation*}
		(x,u)\in P_c
	\end{equation*}
	and so we can conclude that $P_c$ is sequentially compact in $C(T,\RR^N)\times L^1(T,\RR^m)_w$.
	The proof  of Proposition 3.2 is now complete.
\end{proof}

Let $V=\RR^m\times\RR$ and consider the multifunction $\Gamma:T\times\RR^N\rightarrow P_f(V)$ defined by
\begin{eqnarray}\label{eq16}
	\Gamma(t,x) = \{(u,\eta)\in\RR^m\times\RR:u\in U(t,x),\ L(t,x,u)\leq \eta\leq a_{\hat{c}}(t)\}.\\
	\mbox{with}\ \hat{c}>0\ \mbox{from Proposition \ref{prop4} and}\ a_{\hat{c}}(\cdot)\ \mbox{from}\ H(L)(iii). \nonumber
\end{eqnarray}

\begin{prop}\label{prop6}
	If hypotheses $H(U),H(L)$ hold, then
	\begin{itemize}
		\item [(a)] the multifunction $(t,x)\rightarrow\Gamma(t,x)$ is graph measurable;
		\item [(b)] for almost all $t\in T$, the map $ x\rightarrow\Gamma(t,x)$ is usc.
	\end{itemize}
\end{prop}

\begin{proof}
	$(a)$ Let $\overline{B}(t)=\{(x,u,\eta)\in\RR^N\times\RR^m\times\RR:|u|\leq a_0(t),|\eta|\leq a_{\hat{a}}(t)\}$. Then
			\begin{equation*}
				{\rm Gr}\,\Gamma = ({\rm Gr}\, U\times\RR)\cap {\rm epi}\,L\cap \overline{B}.
			\end{equation*}
			Recall that ${\rm epi}\,L=\{(t,x,u,\eta)\in T\times\RR^N\times\RR^m\times\RR:L(t,x,u)\leq\eta\}$. Hypotheses $H(L)(i),\,(ii)$ imply that $(t,x,u)\rightarrow L(t,x,u)$ is measurable. So, it follows that ${\rm epi}\, L\in\mathcal{L}_T\otimes B(\RR^N)\otimes B(\RR^m)\otimes B(\RR)$. Similarly, hypotheses $H(U)(i)\,(ii)$ imply that $(t,x)\rightarrow U(t,x)$ is graph measurable. Therefore
	\begin{equation*}
		{\rm Gr}\,U\times\RR\in\mathcal{L}_T\otimes B(\RR^N)\otimes B(\RR^m)\otimes B(\RR).
	\end{equation*}

	Since $\overline{B}(\cdot)$ is clearly graph measurable, we conclude that
	\begin{equation*}
		{\rm Gr}\,\Gamma\in\mathcal{L}_T\otimes B(\RR^N)\otimes B(\RR^m)\otimes B(\RR).
	\end{equation*}
			
	$(b)$ Suppose that
	$$x_n\rightarrow x\ \mbox{in}\ \RR^N,\ u_n\rightarrow u\ \mbox{in}\ \RR^m,\ \eta_n\rightarrow\eta\ \mbox{in}\ \RR,\ (u_n,\eta_n)\in\Gamma(t,x_n)\ \mbox{for all}\ n\in\NN.$$
			
	We have
	\begin{eqnarray*}
			&&u_n\in U(t,x_n),L(t,x_n,u_n)\leq\eta_n\leq a_{\hat{c}}(t)\\
			&\Rightarrow& u\in U(t,x),\ L(t,x,u)\leq\eta\leq a_{\hat{c}}(t)\\
			&&(\mbox{see hypotheses}\ H(U)(ii),H(L)(ii)),\\
			&\Rightarrow&(u,\eta)\in {\rm Gr}\,\Gamma(t,\cdot),\\
			&\Rightarrow&x\rightarrow\Gamma(t,x)\ \mbox{is usc (see Hu \& Papageorgiou \cite[Proposition 2.23, p.43]{20})}.
	\end{eqnarray*}
The proof  of Proposition 3.3 is now complete.
\end{proof}
	
	We introduce the integrand $\hat{L}:T\times\RR^N\times\RR^m\rightarrow\bar{\RR}=\RR\cup\{+\infty\}$ defined by
	\begin{equation}\label{eq17}
		\hat{L}(t,x,u)=\left\{\begin{array}{ll}
			L(t,x,u)&\mbox{if}\ u\in U(t,x)\\
			+\infty&\mbox{otherwise}.
		\end{array}\right.
	\end{equation}
	
	By $\hat{L}^{**}(t,x,u)$ we denote the second convex conjugate of the function $u\rightarrow\hat{L}(t,x,u)$ (see, for example, Gasinski \& Papageorgiou \cite[p. 512]{17}).
	
	Then the relaxed optimal control problem for (\ref{eq1}) is:
	\begin{equation}\label{eq18}
		\left\{\begin{array}{ll}
			J_r(x,u)=\int^b_0\hat{L}^{**}(t,x(t),u(t))dt\rightarrow \inf=m_r,\\
			\mbox{subject to:}\ (x,u)\in P_c.
		\end{array}\right\}
	\end{equation}
	
	In what follows, given a function $\varphi:\RR^m\rightarrow\bar{\RR}=\RR\cup\{+\infty\}$, we set
	$${\rm dom}\,\varphi=\{u\in\RR^m:\varphi(u)<+\infty\}$$
	(the effective domain of $\varphi$).
\begin{prop}\label{prop7}
	If hypotheses $H(U),H(L)$ hold, then
	\begin{itemize}
		\item[(a)] ${\rm dom}\,\hat{L}^{**}(t,x,\cdot)={\rm conv}\, U(t,x)$ for all $(t,x)\in T\times\RR^N$;
		\item[(b)] $\hat{L}^{**}(t,x,u)=\min\{\eta\in\RR:(u,\eta)\in\overline{\rm conv}\,\Gamma(t,x)\}$ for almost all $t\in T$
		 and all $(x,u)\in\RR^N\times\RR^m$.
	\end{itemize}
\end{prop}
\begin{proof}
	$(a)$ From Proposition 3.2 of Ekeland-Temam \cite[p. 16]{13}, we know that
	\begin{equation}\label{eq19}
		{\rm epi}\,\hat{L}^{**}(t,x,\cdot)=\overline{\rm conv}\, {\rm epi}\,\hat{L}(t,x,\cdot)\ \mbox{for almost all}\ t\in T
		\ \mbox{and all}\ x\in\RR^N.
	\end{equation}
	
	Note that ${\rm dom}\,\hat{L}(t,x,\cdot)=U(t,x)$ (see (\ref{eq17})). So, from (\ref{eq19}) we infer that
	$${\rm dom}\,\hat{L}^{**}(t,x,\cdot)={\rm conv}\, U(t,x)\ \mbox{for almost all}\ t\in T
	\ \mbox{and all}\ x\in\RR^N.$$
	
	$(b)$ Similarly, from (\ref{eq19}) and (\ref{eq16}), we see that
	\begin{eqnarray*}
		&&\hat{L}^{**}(t,x,u)=\inf\{\eta\in\RR:(u,\eta)\in {\rm conv}\,\Gamma(t,x)\}\\
		&&\mbox{for almost all}\ t\in T
		\ \mbox{and all}\ (x,u)\in\RR^N\times\RR^m.
	\end{eqnarray*}
The proof  of Proposition 3.4 is now complete.
\end{proof}

We will also need the following result about the integrable selectors of a graph measurable multifunction, which is actually of independent interest.
\begin{prop}\label{prop8}
	If $G:T\rightarrow P_k(\RR^d)\ (d\geq 1)$ is graph measurable and $S^1_G\neq\emptyset$ and is uniformly integrable, then $\overline{S^1_G}^{||\cdot||_w}=S^1_{{\rm conv}\, G}$.
\end{prop}
\begin{proof}
	From Proposition 3.30 of Hu \& Papageorgiou \cite[p. 185]{20}, we have
	\begin{equation}\label{eq20}
		S^1_{{\rm conv}\, G}=\overline{S^1_{G}}^w\ \mbox{in}\ L^1(\Omega,\RR^N).
	\end{equation}
	
	On the other hand, Proposition 7.16 of Hu \& Papageorgiou \cite[p. 232]{20}, says that given any $\epsilon>0$, we can find $T_{\epsilon}\subseteq T$ closed such that
	$$|T\backslash T_{\epsilon}|_1<\epsilon\ \mbox{and}\ G|_{T_\epsilon}\ \mbox{is}\ \mbox{$h$-continuous}$$
	(here by $|\cdot|$, we denote the Lebesgue measurable on $\RR$). Then we have $G(T_\epsilon)\in P_k(\RR^d)$. Moreover, since $S^1_G$ is uniformly integrable, so is $\overline{\rm conv}\, S^1_G=S^1_{\rm conv\, G}$. Hence $S^1_{{\rm conv}\, G}$ has property $U$ of Bourgain \cite{5} and so applying the theorem of Gutman \cite{19} (see also Hu \& Papageorgiou \cite[Proposition 4.14, p.195]{20}), we infer that on $S^1_{{\rm conv}\, G}$ the $w$-topology and the $||\cdot||_w$-topology coincide. So, from (\ref{eq20}), we have
	$$S^1_{{\rm conv}\, G}=\overline{S^1_G}^{||\cdot||_w}.$$
The proof  of Proposition 3.5 is now complete.
\end{proof}

Using Proposition 3.5, we can obtain
 the following approximation result which is critical in establishing the admissibility of the relaxation (\ref{eq18}).
\begin{prop}\label{prop9}
	If hypotheses $H(A),H(f),H(U),H(L),H_0$ hold and $(\hat{x},\hat{u})\in P_c$, then we can find $\hat{u}_n\in S^1_{U(\cdot,\hat{x}(\cdot))},\ n\in\NN$, such that
	\begin{eqnarray*}
		&&\hat{u}_n\stackrel{||\cdot||_w}{\longrightarrow}\hat{u}\ \mbox{in}\ L^1(T,\RR^m)\\
		&&L(\cdot,\hat{x}_n(\cdot),\hat{u}_n(\cdot))\stackrel{||\cdot||_w}{\longrightarrow}\hat{L}^{**}(\cdot,\hat{x}(\cdot),\hat{u}(\cdot))\ \mbox{in}\ L^1(T).
	\end{eqnarray*}
\end{prop}
\begin{proof}
	Since $(\hat{x},\hat{u})\in P_c$, it follows by Proposition \ref{prop7} that
	\begin{eqnarray*}
		&&(\hat{u}(t),\hat{L}^{**}(t,\hat{x}(t),\hat{u}(t)))\in\overline{\rm conv}\,\Gamma(t,\hat{x}(t))\ \mbox{for almost all}\ t\in T,\\
		&\Rightarrow&\hat{w}(\cdot)=(\hat{u}(\cdot),\hat{L}^{**}(\cdot,\hat{x}(\cdot),\hat{u}(\cdot)))\in S^1_{{\rm conv}\,\Gamma(\cdot,\hat{x}(\cdot))}.
	\end{eqnarray*}
	
	On account of Proposition \ref{prop8}, we can find $w_n\in S^1_{\Gamma(\cdot,\hat{x}(\cdot))}$ such that
	$$w_n\stackrel{||\cdot||_w}{\longrightarrow}w\ \mbox{in}\ L^1(T,\RR^m).$$
	
	However,
	 $w_n(\cdot)=(\hat{u}_n(\cdot),\eta_n(\cdot))$ with $\hat{u}_n\in S^1_{U(\cdot,\hat{x}(\cdot))}$, $L(t,\hat{x}(t),\hat{u}_n(t))\leq\eta_n(t)\leq a_{\hat{c}}(t)$ for almost all $t\in T$, and all $n\in\NN$. Therefore
	\begin{eqnarray*}
		&&\hat{u}_n\stackrel{||\cdot||_w}{\longrightarrow}\hat{u}\ \mbox{in}\ L^1(T,\RR^m)\\
		\mbox{and}&&\sup\limits_{t\in T}\int^t_0(\eta_n(s)-\hat{L}^{**}(s,\hat{x}(s),\hat{u}_n(s)))ds\rightarrow 0,\\
		&\Rightarrow&\sup\limits_{t\in T}\int^t_0(L(s,\hat{x}(s),\hat{u}_n(s))-\hat{L}^{**}(s,\hat{x}(s),\hat{u}_n(s)))ds\rightarrow 0,\\
		&\Rightarrow&L(\cdot,\hat{x}(\cdot),\hat{u}_n(\cdot))\stackrel{||\cdot||_w}{\longrightarrow}\hat{L}^{**}(\cdot,\hat{x}(\cdot),\hat{u}(\cdot))\ \mbox{in}\ L^1(\Omega).
	\end{eqnarray*}
The proof  of Proposition 3.6 is now complete.
\end{proof}

Now we can establish the density of $P$ in $P_c$ in the space $C(T,\RR^N)\times L^1(T,\RR^m)_w$.
\begin{prop}\label{prop10}
	If hypotheses $H(A),H(f),H(U),H(L),H_0$ hold and $(\hat{x},\hat{u})\in P_c$, then we can find $\{(\hat{x}_n,\hat{v}_n)\}_{n\geq 1}\subseteq P$ such that
	\begin{eqnarray*}
		&&\hat{x}_n\rightarrow\hat{x}\ \mbox{in}\ C(T,\RR^N),\\
		&&v_n \xrightarrow[||\cdot||_w]{w} \hat{u}\ \mbox{in}\ L^1(T,\RR^m),\\
		&&L(\cdot,\hat{x}_n(\cdot),\hat{v}_n(\cdot))\stackrel{||\cdot||_w}{\longrightarrow}\hat{L}^{**}(\cdot,\hat{x}(\cdot),\hat{u}(\cdot)).
	\end{eqnarray*}
\end{prop}
\begin{proof}
	According to Proposition \ref{prop9}, we can find $\{\hat{u}_n\}_{n\geq 1}\subseteq S^1_{U(\cdot,\hat{x}(\cdot))}$ such that
	\begin{equation}\label{eq21}
		\hat{u}_n\stackrel{||\cdot||_w}{\longrightarrow}\hat{u}\ \mbox{in}\ L^1(T,\RR^m)\ \mbox{and}\ L(\cdot,\hat{x}(\cdot),\hat{u}_n(\cdot))\stackrel{||\cdot||_w}{\longrightarrow}\hat{L}^{**}(\cdot,\hat{x}(\cdot),\hat{u}(\cdot))\ \mbox{in}\ L^1(T).
	\end{equation}
	
	Since $\{\hat{u}\}_{n\geq 1}\subseteq L^1(T,\RR^m)$ is bounded (see hypothesis $H(U)(iii)$), 
	it follows by Proposition \ref{prop3} that
	\begin{equation}\label{eq22}
		\hat{u}_n\stackrel{w}{\rightarrow}\hat{u}\ \mbox{in}\ L^1(T,\RR^m).
	\end{equation}
	
	Consider the multifunction $V_n:T\times\RR^N\rightarrow 2^{\RR^m}\backslash\{\emptyset\}$ defined by
	\begin{equation}\label{eq23}
		V_n(t,x)=\{\xi\in\RR^m:|\hat{u}_n(t)-\xi|\leq k(t)|\hat{x}(t)-x|+\frac{1}{n},\ \xi\in U(t,x)\}
	\end{equation}
	(see hypothesis $H(U)(ii)$).

We consider the following control system:
	\begin{equation}\tag{$24_n$}\label{eq24n}
		\left\{\begin{array}{l}
			-x'(t)\in A(x(t))+f(t,x(t))v(t)\ \mbox{for almost all}\ t\in T=[0,b],\\
			x(0)=x_0,\ v\in S^1_{V_n(\cdot,x(\cdot))}\,.
		\end{array}\right\}
	\end{equation}
	
	Reasoning as in the proof of Proposition \ref{prop4}, we can show that for every $n\in\NN$, problem \eqref{eq24n} has admissible state-control pairs. So, let $(\hat{x}_n,\hat{v}_n)$ be such a pair for \eqref{eq24n}. Then
	$$(\hat{x}_n,\hat{v}_n)\in P\ \mbox{for all}\ n\in\NN.$$
	
	As in the proof of Proposition \ref{prop4}, using the theorem of Baras \cite{2}, we see that
	$$\{\hat{x}_n\}_{n\geq 1}\subseteq C(T,\RR^N)\ \mbox{is relatively compact}.$$
	\stepcounter{equation}
	So, we may assume that
	\begin{equation}\label{eq25}
		\hat{x}_n\rightarrow\hat{x}^*\ \mbox{in}\ C(T,\RR^N).
	\end{equation}
	
	We have
	\begin{eqnarray*}
		&&-\hat{x}'_n(t)\in A(\hat{x}_n(t))+f(t,\hat{x}_n(t))v_n(t)\ \mbox{for almost all}\ t\in T,\ \hat{x}_n(0)=x_0,\ n\in\NN,\\
		&&-\hat{x}'(t)\in A(\hat{x}(t))+f(t,\hat{x}(t))\hat{u}(t)\ \mbox{for almost all}\ t\in T,\ \hat{x}(0)=x_0.
	\end{eqnarray*}
	
	Exploiting the monotonicity of $A(\cdot)$, we obtain
	\begin{eqnarray}\label{eq26}
		&&(\hat{x}'_n(t)-\hat{x}'(t),\hat{x}_n(t)-\hat{x}(t))_{\RR^N}+(f(t,\hat{x}_n(t))\hat{v}_n(t)-f(t,\hat{x}(t))\hat{u}(t),\hat{x}_n(t)-\nonumber\\
		&&\hat{x}(t))_{\RR^N}\leq 0\ \mbox{for almost all}\ t\in T,\nonumber\\
		&\Rightarrow&\frac{1}{2}\frac{d}{dt}|\hat{x}_n(t)-\hat{x}(t)|^2+
(f(t,\hat{x}_n((t)))\hat{v}_n(t)-f(t,\hat{x}(t))\hat{u}(t),\hat{x}_n(t)-\hat{x}(t))_{\RR^N}\leq 0\nonumber\\
		&&\mbox{for almost all}\ t\in T,\nonumber\\
		&\Rightarrow&\frac{1}{2}|\hat{x}_n(t)-\hat{x}(t)|^2+\int^1_0(f(s,\hat{x}_n(s))\hat{v}_n(s)-
f(s,\hat{x}(s))\hat{u}(s),\hat{x}_n(s)-\nonumber\\
		&&\hat{x}(s))_{\RR^N}ds\leq 0\ \mbox{for all}\ n\in\NN.
	\end{eqnarray}
	
	We estimate the integral on the left-hand side of (\ref{eq26}). Then
	\begin{eqnarray}\label{eq27}
		&&\left|\int^t_0(f(s,\hat{x}_n(s))\hat{v}_n(s)-f(s,\hat{x}(s))\hat{u}(s),\hat{x}_n(s)-\hat{x}(s))_{\RR^N}ds\right|\nonumber\\
		&\leq&\left|\int^t_0(f(s,\hat{x}_n(s))\hat{v}_n(s)-f(s,\hat{x}_n(s))\hat{u}_n(s),\hat{x}_n(s)-\hat{x}(s))_{\RR^N}ds\right|\nonumber\\
		&+&\left|\int^t_0(f(s,\hat{x}_n(s))\hat{u}_n(s)-f(s,\hat{x}_n(s))\hat{u}(s),\hat{x}_n(s)-\hat{x}(s))_{\RR^N}ds\right|\nonumber\\
		&+&\left|\int^t_0(f(s,\hat{x}_n(s))\hat{u}(s)-f(s,\hat{x}(s))\hat{u}(s),\hat{x}_n(s)-\hat{x}(s))_{\RR^N}ds\right|.
	\end{eqnarray}
	
	We examine each summand on the right-hand side of (\ref{eq27}). We have
	\begin{eqnarray}\label{eq28}
		&&\left|\int^t_0(f(s,\hat{x}_n(s))\hat{v}_n(s)-f(s,\hat{x}_n(s))\hat{u}_n(s),\hat{x}_n(s)-\hat{x}(s))_{\RR^N}ds\right|\nonumber\\
		&=&\left|\int^t_0(\hat{v}_n(s)-\hat{u}_n(s),f(s,\hat{x}_n(s))^*(\hat{x}_n(s)-\hat{x}(s)))_{\RR^N}ds\right|\nonumber\\
		&\leq&\left|\int^t_0|\hat{v}_n(s)-\hat{u}_n(s)|\ ||f(s,\hat{x}_n(s))^*||_{\mathcal{L}}|\hat{x}_n(s)-\hat{x}(s)|ds\right|\nonumber\\
		&=&\left|\int^t_0|\hat{v}_n(s)-\hat{u}_n(s)|\ ||f(s,\hat{x}_n(s))||_{\mathcal{L}}|\hat{x}_n(s)-\hat{x}(s)|ds\right|\nonumber\\
		&\leq&\int^t_0[k(s)|\hat{x}_n(s)-\hat{x}(s)|+\frac{1}{n}]a_1(s)|\hat{x}_n(s)-\hat{x}(s)|ds\nonumber\\
		&&\mbox{for some}\ a_1\in L^1(T),\ \mbox{all}\ n\in\NN\ (\mbox{see (\ref{eq23}), (\ref{eq25}) and hypothesis}\ H(f)(iii))\nonumber\\
		&\leq&\int^t_0a_2(s)|\hat{x}_n(s)-\hat{x}(s)|^2ds+\frac{c_2}{n}||\hat{x}_n-\hat{x}||_{C(T,\RR^N)}\\
		&&\mbox{for some}\ a_2\in L^1(T),c_2>0\ \mbox{and all}\ n\in\NN.\nonumber
	\end{eqnarray}
	
	For the second summand we have
	\begin{eqnarray}\label{eq29}
		&&\left|\int^t_0(f(s,\hat{x}_n(s))\hat{u}_n(s)-f(s,\hat{x}_n(s))\hat{u}(s),\hat{x}_n(s)-\hat{x}(s))_{\RR^N}ds\right|\nonumber\\
		&=&\left|\int^t_0(\hat{u}_n(s)-\hat{u}_n(s),f(s,\hat{x}_n(s))^*(\hat{x}_n(s)-\hat{x}(s))_{\RR^m}ds\right|.
	\end{eqnarray}
	
	We know that bounded sets in $L^{\infty}(T,\RR^m)$ furnished with the $w^*$-topology, are metrizable (see \cite[p. 230]{23}). Hence hypothesis $H(U)(iii)$ and (\ref{eq25}) imply that
	\begin{equation}\label{eq30}
		\hat{u}_n\stackrel{w^*}{\rightarrow}\hat{u}\ \mbox{in}\ L^{\infty}(T,\RR^m).
	\end{equation}
	
	Also note that
	$$f(s,\hat{x}_n(s))^*(\hat{x}_n(s)-\hat{x}(s))\rightarrow f(s,\hat{x}^*(s))^*(\hat{x}^*(s)-\hat{x}(s))\ \mbox{for almost all}\ s\in T.$$
	
	So, by the Lebesgue dominated convergence theorem, we have
	\begin{equation}\label{eq31}
		f(\cdot,\hat{x}_n(\cdot))^*(\hat{x}_n-\hat{x})(\cdot)\rightarrow f(\cdot,\hat{x}^*(\cdot))(\hat{x}^*-\hat{x})(\cdot)\ \mbox{in}\ L^1(T,\RR^m).
	\end{equation}
	
	From (\ref{eq29}), (\ref{eq30}), (\ref{eq31}) it follows that
	\begin{equation}\label{eq32}
		\left|\int^t_0(f(s,\hat{x}_n(s))\hat{u}_n(s)-f(s,\hat{x}_n(s))\hat{u}(s),\hat{x}_n(S)-\hat{x}(s))_{\RR^N}ds\right|\rightarrow 0.
	\end{equation}
	
	Finally, for the third summand, we have
	\begin{eqnarray}\label{eq33}
		&&\left|\int^t_0(f(s,\hat{x}_n(s))\hat{u}(s)-f(s,\hat{x}_n(s))\hat{u}(s),\hat{x}_n(s)-\hat{x}(s))_{\RR^N}ds\right|\nonumber\\
		&\leq&c_3\int^t_0 l_r(s)|\hat{x}_n(s)-\hat{x}(s)|^2ds\ \mbox{with}\ r=\sup\limits_{n\geq 1}||\hat{x}_n||_{C(T,\RR^N)}<\infty,\ \mbox{for some}\ c_3>0\\
		&&(\mbox{see hypotheses}\ H(f)(ii),H(U)(iii)).\nonumber
	\end{eqnarray}
	
	We return to (\ref{eq27}), pass to the limit as $n\rightarrow\infty$, and use (\ref{eq28}), (\ref{eq32}), (\ref{eq33}) and (\ref{eq25}). Then
	\begin{eqnarray}\label{eq34}
		&&\limsup\limits_{n\rightarrow\infty}\left|\int^t_0(f(s,\hat{x}_n(s))\hat{u}_n(s)-f(s,\hat{x}(s))\hat{u}(s),\hat{x}_n(s)-\hat{x}(s))_{\RR^N}ds\right|\nonumber\\
		&\leq&\int^t_0a_3(s)|\hat{x}^*(s)-\hat{x}(s)|^2ds\ \mbox{for some}\ a_3\in L^1(T)\ (\mbox{see (\ref{eq25})}).
	\end{eqnarray}
	
	So, if in (\ref{eq26}) we pass to the limit as $n\rightarrow\infty$, and use (\ref{eq25}) and (\ref{eq34}), we get
	\begin{eqnarray*}
		&&|\hat{x}^*(t)-\hat{x}(t)|^2\leq\int^t_02a_3(s)|\hat{x}^*(s)-\hat{x}(s)|^2ds\ \mbox{for all}\ t\in T,\\
		&\Rightarrow&\hat{x}^*=\hat{x}\ (\mbox{by Gronwall's inequality}).
	\end{eqnarray*}
	
	Therefore for the original sequence we have
	\begin{equation}\label{eq35}
		\hat{x}_n\rightarrow\hat{x}\ \mbox{in}\ C(T,\RR^N).
	\end{equation}
	
	From (\ref{eq23}) and (\ref{eq35}), we see that
	\begin{eqnarray*}
		&&||\hat{u}_n-\hat{v}_n||_1\rightarrow 0,\\
		&\Rightarrow&\hat{v}_n\xrightarrow[||\cdot||_w]{w}\hat{u}\ \mbox{in}\ L^1(T,\RR^N)\ (\mbox{see (\ref{eq21}) and (\ref{eq22})}).
	\end{eqnarray*}
	
	Finally, recall (see (\ref{eq21})) that
	\begin{equation}\label{eq36}
		L(\cdot,\hat{x}(\cdot),\hat{u}_n(\cdot))\stackrel{||\cdot||_w}{\longrightarrow}\hat{L}^{**}(\cdot,\hat{x}(\cdot),\hat{u}(\cdot))\ \mbox{in}\ L^1(T).
	\end{equation}
	
	On account of hypotheses $H(L)(ii)$ and (\ref{eq23}), (\ref{eq35}), we have
	\begin{equation}\label{eq37}
		L(\cdot,\hat{x}_n(\cdot),\hat{v}_n(\cdot))-L(\cdot,\hat{x}(\cdot),\hat{u}(\cdot))\stackrel{||\cdot||_w}{\longrightarrow} 0.
	\end{equation}
	
	From (\ref{eq36}) and (\ref{eq37}) we obtain $$L(\cdot,\hat{x}_n(\cdot),\hat{v}_n(\cdot))\stackrel{||\cdot||_w}{\longrightarrow}\hat{L}^{**}(\cdot,\hat{x}(\cdot),\hat{v}(\cdot))\ \mbox{in}\ L^1(T,\RR).$$
The proof  of Proposition 3.7 is now complete.
\end{proof}

Now we are ready to show that our first relaxation method which produces problem (\ref{eq18}), is admisible.
\begin{theorem}\label{th11}
	If hypotheses $H(A),H(f),H(U),H(L),H_0$ hold, then the following properties hold:
	\begin{itemize}
		\item[(a)] there exists $(\hat{x},\hat{u})\in P_c$ such that
		$$J_r(\hat{x},\hat{u})=m_r;$$
		\item[(b)] $m_r=m$;
		\item[(c)] there exists a sequence $\{(\hat{x}_n,\hat{u}_n)\}_{n\geq 1}\subseteq P$ such that
		\begin{eqnarray*}
			&&J(\hat{x}_n,\hat{u}_n)\downarrow m_r\\
			&\mbox{and}&\hat{x}_n\rightarrow\hat{x}\ \mbox{in}\ C(T,\RR^N)\ \mbox{and}\ \hat{u}_n\xrightarrow[||\cdot||_w]{w}\hat{u}\ \mbox{in}\ L^1(T,\RR^m)
		\end{eqnarray*}
	\end{itemize}
\end{theorem}
\begin{proof}
	$(a)$ Let $\{(\hat{y}_n,\hat{v}_n)\}_{n\geq 1}\subseteq P_c$ be a minimizing sequence for the relaxed problem (\ref{eq18}). So, we have
	$$J_r(\hat{y}_n,\hat{v}_n)\downarrow m_r.$$
	
	We know that
	$$\{(\hat{y}_n,\hat{v}_n)\}_{n\geq 1}\subseteq C(T,\RR^N)\times L^1(T,\RR^m)_w\ \mbox{is relatively compact}$$
	(see Proposition \ref{prop5}). So, we may assume that
	\begin{equation}\label{eq38}
		\hat{y}_n\rightarrow\hat{x}\ \mbox{in}\ C(T,\RR^N)\ \mbox{and}\ \hat{v}_n\stackrel{w}{\rightarrow}\hat{u}\ \mbox{in}\ L^1(T,\RR^m).
	\end{equation}
	
	Then from (\ref{eq38}) and Proposition 3.151 of Gasinski \& Papageorgiou \cite[p. 441]{18}, we have
	\begin{equation}\label{eq39}
		J_r(\hat{x},\hat{u})\leq\liminf\limits_{n\rightarrow\infty}J_r(\hat{y}_n,\hat{v}_n)=m_r.
	\end{equation}
	
	However, from (\ref{eq38}) and Proposition \ref{prop5}, we have
	\begin{eqnarray*}
		&&(\hat{x},\hat{u})\in P_c,\\
		&\Rightarrow&J_r(\hat{x},\hat{u})=m_r\ (\mbox{see (\ref{eq39})}).
	\end{eqnarray*}
	
	$(b)$ and $(c)$: From Proposition \ref{prop10}, we know that there exists a sequence $\{(\hat{x}_n,\hat{u}_n)\}_{n\geq 1}\subseteq P$ such that
	\begin{eqnarray}\label{eq40}
		&&\hat{x}_n\rightarrow\hat{x}\ \mbox{in}\ C(T,\RR^N),\ \hat{u}_n\xrightarrow[||\cdot||_w]{w}\hat{u}\ \mbox{in}\ L^1(T,\RR^m),\nonumber\\
		&&L(\cdot,\hat{x}(\cdot),\hat{u}_n(\cdot))\stackrel{||\cdot||_w}{\longrightarrow}\hat{L}^{**}(\cdot,\hat{x}(\cdot),\hat{u}(\cdot))\ \mbox{in}\ L^1(T).
	\end{eqnarray}
	
	We have
	\begin{eqnarray*}
		&&|J_r(\hat{x}_n,\hat{u}_n)-J_r(\hat{x},\hat{u})|\\
		&\leq&||L(\cdot,\hat{x}_n(\cdot),\hat{u}_n(\cdot))-\hat{L}^{**}(\cdot,\hat{x}(\cdot),\hat{u}(\cdot))||_w\rightarrow 0\ (\mbox{see (\ref{eq40})}),\\
		&\Rightarrow&J(\hat{x},\hat{u}_n)\rightarrow J_r(\hat{x},\hat{u})=m_r.
	\end{eqnarray*}
	
	We know that
	\begin{eqnarray*}
		&&m\leq J(\hat{x}_n,\hat{u}_n)\ \mbox{for all}\ n\in\NN,\\
		&\Rightarrow&m\leq m_r,\\
		&\Rightarrow&m=m_r\ (\mbox{since we always have}\ m_r\leq m).
	\end{eqnarray*}
The proof  of Theorem 3.8 is now complete.
\end{proof}

\section{Second relaxation method}

In this section we present an alternative relaxation method based on Young measures.

So, now the control constraint set for the relaxed system is
$$\Sigma(t,x)=\{\mu\in R(T,\bar{B}_M):\mu(U(t,x))=1\},$$
where $M=||a_0||_{\infty}$ (see hypothesis $H(U)(iii)$) and $\bar{B}_M=\{v\in \RR^m:|v|\leq M\}$. Then, the new relaxed optimal control problem, is the following:
\begin{equation}\label{eq41}
	\left\{\begin{array}{l}
		\hat{J}_r(x,\lambda)=\int^b_0\int_{\bar{B}}L(t,x(t),u)\lambda(t)(du)dt\rightarrow \inf=\hat{m}_r,\\
		\mbox{subject to:}\ -x'(t)\in A(x(t))+\int_{\bar{B}_M}f(t,x(t))u\lambda(t)(du)\ \mbox{for a.a.}\ t\in T=[0,b]\\
		x(0)=x_0,\ \lambda\in S_{\Sigma(\cdot,x(\cdot))},
	\end{array}\right\}
\end{equation}
where
$$S_{\Sigma(\cdot,x(\cdot))}=\{\lambda:T\rightarrow P(\bar{B}_M):\lambda(\cdot)\ \mbox{is measurable},\ \lambda(t)\in\Sigma(t,x(t))\ \mbox{for almost all}\ t\in T\}.$$

Recall that on $P(\bar{B}_M)$ we consider the narrow topology $\tau_n$.

The orientor field for this new convexified control system is
$$\hat{F}(t,x)=\left\{\int_{\bar{B}_M}f(t,x)u\lambda(t)(du):\lambda\in S_{\Sigma(\cdot,x)}\right\}.$$

So, if we denote by $\hat{P}_c$ the set of admissible state-control pairs for (\ref{eq41}), then $\hat{P}_c$ is compact in $C(T,\RR^N)\times {\mathcal Y}(T\times \bar{B}_M)_{\tau^y_n}$.

We show that this new relaxed optimal control problem is equivalent to (\ref{eq18}) and hence also  admissible.
\begin{theorem}\label{th12}
	If hypotheses $H(A),H(f),H(U),G(L),H_0$ hold, then there exists $(\hat{x},\hat{\lambda})\in\hat{P}_c$ such that
	$$\hat{J}_r(\hat{x},\hat{\lambda})=\hat{m}_r=m=m_r$$
	and we can find a sequence $\{(\hat{x}_n,\hat{u}_n)\}_{n\geq 1}\subseteq P$ such that
	$$\hat{x}_n\rightarrow\hat{x}\ \mbox{in}\ C(T,\RR^N)\ \mbox{and}\ \delta_{\hat{u}_n}\stackrel{\tau^y_n}{\longrightarrow}\hat{\lambda}\ \mbox{in}\ R(T,\bar{B}_M).$$
\end{theorem}

\begin{proof}
	Let $\{(\hat{x}_n,\hat{u}_n)\}_{n\geq 1}\subseteq\hat{P}_c$ be a minimizing sequence for problem (\ref{eq41}). Evidently, $\{\hat{\lambda}_n\}_{n\geq 1}\subseteq R(T,\bar{B}_M)$ is tight, so by Prokhorov's theorem (see Papageorgiou \& Winkert \cite[Theorem 4.7.22, p.393]{23}), we have at least for a subsequence that
	\begin{equation}\label{eq42}
		\hat{\lambda}\stackrel{\tau^y_n}{\longrightarrow}\hat{\lambda}\ \mbox{in}\ R(T,\bar{B}_M).
	\end{equation}
	
	Let $\hat{x}_n\in W^{1,1}((0,b),\RR^N)=AC^1(T,\RR^N)$ be the relaxed state generated by the relaxed control $\hat{\lambda}_n\in S_{\Sigma(\cdot,x_n(\cdot))}$. We know that
	$$\{\hat{x}_n\}_{n\geq 1}\subseteq C(T,\RR^N)\ \mbox{is relatively compact (see the proof of Proposition \ref{prop4})}.$$
	
	So, we may assume that
	\begin{equation}\label{eq43}
		\hat{x}_n\rightarrow \hat{x}\ \mbox{in}\ C(T,\RR^N).
	\end{equation}
	
	We have
	$$(\hat{x},\hat{u})\in\hat{P}_c\ (\mbox{see (\ref{eq42}), (\ref{eq43})}).$$
	
	Note that
	\begin{eqnarray*}
		\hat{J}_r(\hat{x}_n,\hat{\lambda}_n)&=&\int^b_0\int_{\bar{B}_M}L(t,\hat{x}_n(t),u)\hat{\lambda}_n(t)(du)dt\\
		&=&\int^b_0\int_{\RR^N\times\bar{B}_M}L(t,x,u)(\delta_{\hat{x}_n(t)}\otimes\hat{\lambda}_n(t))(dx,du)dt.
	\end{eqnarray*}
	
	By the Fiber Product Lemma (see Valadier \cite[Lemma 11]{26}), we have
	$$\delta_{\hat{x}_n(\cdot)}\otimes\hat{\lambda}_n\stackrel{\tau^y_n}{\longrightarrow}\delta_{\hat{x}(\cdot)}\otimes\hat{\lambda}\ (\mbox{see (\ref{eq42}), (\ref{eq43})}).$$
	
	Then on account of hypotheses $H(L)$ and the definition of the narrow topology $\tau^y_n$, we have
	\begin{eqnarray*}
		\hat{m}_r&=&\lim\limits_{n\rightarrow\infty}\hat{J}_r(\hat{x}_n,\hat{\lambda}_n)\\
		&=&\lim\limits_{n\rightarrow\infty}\int^b_0\int_{\bar{B}_M}L(t,x_n(t),u)\hat{\lambda}_n(t)(du)dt\\
		&=&\int^b_0\int_{\bar{B}_M}L(t,\hat{x}(t),u)\hat{\lambda}(t)(du)dt\\
		&=&\hat{J}_r(\hat{x},\hat{\lambda})\geq\hat{m}_r\ (\mbox{since}\ (\hat{x},\hat{\lambda})\in\hat{P}_c),\\
		\Rightarrow&&\hat{J}_r(\hat{x},\hat{\lambda})=\hat{m}_r.
	\end{eqnarray*}
	
	By the density of the Dirac Young measures in ${\mathcal Y}(T\times\bar{B}_M)$ with the Young narrow topology $\tau^y_n$ (see Valadier \cite[Proposition 8]{26}),  we can find $\{\hat{u}_n\}_{n\geq 1}\subseteq S^1_{U(\cdot,\hat{x}(\cdot))}$ such that
	\begin{equation}\label{eq44}
		\delta_{\hat{u}_n(\cdot)}\stackrel{\tau^y_n}{\longrightarrow}\hat{\lambda}.
	\end{equation}
	
	As in the proof of Proposition \ref{prop10}, we consider the multifunction $(t,x)\rightarrow V_n(t,x)$ defined by
	$$V_n(t,x)=\left\{\xi\in\RR^m:|\hat{u}_n(t)-\xi|\leq k(t)|\hat{x}(t)-x|+\frac{1}{n},\ \xi\in U(t,x)\right\}.$$
	
	We use this as control constraint multifunction and consider the control system \eqref{eq24n}. We can find admissible state-control pairs $(\hat{x}_n,\hat{v}_n),\ n\in\NN$ for this system such that
	\begin{equation}\label{eq45}
		\hat{x}_n\rightarrow\hat{x}\ \mbox{in}\ C(T,\RR^N),\ \hat{u}_n-\hat{v}_n\xrightarrow[||\cdot||_1]{\text{a.e.}} 0.
	\end{equation}
	
	Using (\ref{eq44}) and (\ref{eq45}), we have
	\begin{eqnarray*}
		&&J(\hat{x}_n,\hat{v}_n)\rightarrow\hat{J}_r(\hat{x},\hat{\lambda})=\hat{m}_r,\\
		&\Rightarrow&m\leq\hat{m}_r,\\
		&\Rightarrow&m=\hat{m}_r=m_r.
	\end{eqnarray*}
The proof  of Theorem 4.1 is now complete.
\end{proof}

\medskip
{\bf Acknowledgments.} This research was supported by the Slovenian Research Agency grants
P1-0292, J1-8131, N1-0114, N1-0064, and N1-0083.

\end{document}